\documentclass[12pt,a4paper]{article}
\usepackage{graphicx}
\usepackage{amsmath,amscd,amssymb,latexsym,upref} 
\usepackage{color,enumitem}
\usepackage{bm}
\usepackage{url}

\newcommand{\Hil}[0]{
\mathcal{H} 
}
\newcommand{\norm}[2]{
\left\| #2 \right\|_{#1}
}

\newcommand{\NN}[0]{
{\mathbb{N}}
}

\newcommand{\RRd}[0]{
{\mathbb{R}^d}
}
\newcommand{\BL}[0]{
{\mathcal B}
}

\renewcommand{\d}{\operatorname{d}\!}
\newtheorem{theorem}{Theorem}[section]
\newtheorem{definition}{Definition}[section]

\newtheorem{conj.}[theorem]{Conjecture}
\newtheorem{Bsp.}{Example}[section]

\newenvironment{proof}{\noindent \bf Proof: \rm}{$ \hspace{\stretch{1}} \Box $

\vspace{5mm}}

\newcommand{\Mat}{\mathcal M}

\newcommand{\range}[1]{\operatorname{ran}\left( #1 \right)} 
\newcommand{\identity}[1]{\operatorname{id}_{ #1 }}
\newcommand{\kernel}[1]{\ker\left( #1 \right)}

\newcommand{\diag}[1]{{\operatorname{diag}} \left( #1 \right)}

\definecolor{darkviolet}{rgb}{0.58,0,0.83} 

\title{Frames for the solution of operator equations in Hilbert 
spaces with fixed dual pairing} 
\author{Peter Balazs$^\star$ and Helmut Harbrecht$^\dagger$}

\begin{document}

\maketitle
\thanks{${}^\star$Austrian Academy of Sciences, Acoustics Research Institute, 
Reichsratsstrasse 17, 1010 Vienna, Austria, Peter.Balazs@oeaw.ac.at, \url{https://www.kfs.oeaw.ac.at/balazs}, \url{https://orcid.org/0000-0003-4939-0831};

${}^\dagger$University of Basel, Department of Mathematics and Computer Science, 
Spiegelgasse 1, 4051 Basel, Switzerland, Helmut.Harbrecht@unibas.ch, \url{cm.dmi.unibas.ch}, \url{https://orcid.org/0000-0003-0093-5706}}

\begin{abstract} For the solution of operator equations, 
Stevenson introduced a definition 
of frames, where a Hilbert space and its dual are {\em not} 
identified. This means that the Riesz isomorphism is not used 
as an identification, which, for example, does not make sense 
for the Sobolev spaces $H_0^1(\Omega)$ and 
$H^{-1}(\Omega)$. In this article, we are going to 
revisit the concept of Stevenson frames and introduce 
it for Banach spaces. This is equivalent to $\ell^2$-Banach 
frames. It is known that, if such a system exists, by defining 
a new inner product and using the Riesz isomorphism, the 
Banach space is isomorphic to a Hilbert space. In this article, 
we deal with the contrasting setting, where $\Hil$ and $\Hil'$ 
are not identified, and equivalent norms are distinguished, 
and show that in this setting the investigation of $\ell^2$-Banach 
frames make sense.
\\[2mm]
\end{abstract}

\section{Introduction}
The standard definition of frames found first in the paper 
by Duffin and Schaefer \cite{duffschaef1} is the following:
\begin{equation} \label{sec:fram1}
\norm{\Hil}{f} \approx \norm{\ell^2}{\left< f , \psi_k \right>_{\Hil}} 
\text{ for all } f \in \Hil.  \end{equation}
Here, $x \approx y$ means that there are constants 
$0 < A \le B < \infty$ such that $A \cdot x \le y \le B \cdot x$.

This concept led to a lot of theoretical work, see e.g.~\cite{al95,behe90,%
ole1n,pesmha17,pjilstoev14}, but has been used also extensively in 
signal processing \cite{boelc1}, quantum mechanics \cite{Gaz09}, 
acoustics \cite{framepsycho16} and various other fields. 

Frames can be used also to represent operators.  
For the numerical solution of operator equations, the (Petrov-) Galerkin 
scheme  \cite{sauter2010boundary} is used, where operators are represented 
by $\left< O \psi_k , \phi_l \right>_{k,l\in K}$, called the \emph{stiffness} or \emph{system matrix}. 
The collection $\Psi = \left(\psi_k\right)_{k \in K}$ consists of 
the ansatz functions, the collection $\Phi = \left(\phi_k\right)_{k\in K}$ 
are the test functions. If $\Psi$ and $\Phi$ live in the same space, 
this is called Galerkin scheme, otherwise it is called Petrov-Galerkin 
scheme.

In finite and boundary element approaches, bases were used 
\cite{DahSch99a,Gauletal03}, but also frames have been applied, 
e.g.\ in \cite{gri94,groobe15,harbr08,os94,Stevenson03}. Recently, 
such operator representations got also a more theoretical treatment 
\cite{xxlframoper1,xxlgro14,xxlriek11}.

In numerical applications, it is often advantageous to have self-adjoint 
matrices, e.g.\ for Krylov subspace methods, which necessitates to use 
the same sequence for the discretization at both sides, i.e.\ investigating 
$\left< O \psi_k , \psi_l \right>_{k,l \in K}$. Note that this matrix is 
self-adjoint if $O$ is, and semi-positive if $O$ is. Positivity is in 
general not preserved, only if a system without redundancy is 
used, i.e.\ a Riesz sequence. Partial differential operators 
are typically operators of the form $O: \Hil \rightarrow \Hil'$, 
while boundary integral operators might also be smoothing 
operators which map in accordance with $O: \Hil' \rightarrow \Hil$. 
One possible solution is to work with Gelfand triples i.e, 
$\Hil \subset \Hil_0 \subset \mathcal\Hil'$. This is 
explicitly done for the concept of Gelfand frames \cite{dafora05}. 

Another possibility 
is the following, introduced by Stevenson in 
\cite{Stevenson03} and used e.g.\ in \cite{harbr08}:
A collection $\Psi = \left(\psi_k\right)_{k \in K} \subset \Hil$ 
is called a (Stevenson) frame for $\Hil$, if
\begin{equation} \label{eq:stevframintr1}
\norm{\Hil'}{f} \approx \norm{\ell^2}{\left< f , \psi_k \right>_{\Hil',\Hil}}
\ \text{for all}\ f \in \Hil'.  
\end{equation} 
Note the difference to the definition \eqref{sec:fram1} by Duffin 
and Schaefer, which is significant only if the Riesz isomorphism 
is not employed. Here, the Gelfand triple is only implicitly used and, if the fully 
general setting is used, the density of the spaces is not required.

Clearly, the definitions \eqref{sec:fram1} and \eqref{eq:stevframintr1} are equivalent by the Riesz isomorphism.
On the other hand, if the isomorphism $\Hil \cong \Hil'$ is not 
considered, but another one is utilized, for example, considering 
the triple $\Hil \subset \Hil_0 \subset \Hil'$, then the Riesz 
isomorphism is usually used as an identification on the pivot 
space $\Hil_0 \cong  \Hil_0'$ , and therefore $\Hil$ and $\Hil'$ cannot 
be considered to be equal.

In this article, we consider the original definition by Stevenson and 
re-investigate in full detail all the derivation to ensure that the Riesz 
identification does not 'creep in' again.

On a more theoretical level, let us consider Banach frames 
\cite{Casazza2005710,chst03,gr91}. Thus, we consider a Banach 
space $X$, a sequence space $X_d$, and a sequence 
$\Psi \subset X'$. This is a $X_d$-frame if 
$$ 
\norm{X}{f} \approx \norm{X_d}{\left< f , \psi_k \right>_{X,X'}} 
\ \text{for all}\ f \in X.
$$
It is called a Banach frame if a reconstruction operator exists, 
i.e.\ there exists $R: X_d \rightarrow X$ with $R \left( \psi_k(f) \right) = f$ 
for all $f \in X$.

In this setting, $\ell_2$-frames were not considered to be interesting as they are 
isomorphic to Hilbert frames, see e.g.\ \cite[Proposition 3.10]{stoev09}: 
Let $\Psi$ be a $\ell^2$-frame for $X$. Then, $X$ can be equipped 
with an inner product $\left<f,g\right>_{X}=\left< C_\Psi f, C_\Psi g 
\right>_{\ell^2}$, becoming a Hilbert space, and $\Psi$ is a (Hilbert) 
frame for $X$. The proof uses the Riesz isomorphism $\Hil \cong \Hil'$ 
in the last line. But if a context is considered, where this isomorphism 
cannot be applied, like for example a Gelfand triple setting, suddenly 
the concept of $\ell_2$-frames might become non-trivial again, and 
the concept of Stevenson frames is different to a frame. In this 
article, we investigate this approach.

The rest of this article is structured as follows. In Section~\ref{sec:solvopeq0},
we motivate Gelfand triples $\mathcal{H}'\subset\mathcal{H}_0\subset
\mathcal{H}$ by a simple example arising from the variational formulation
of second order elliptic partial differential equations. Section~\ref{sec:intronot0}
then provides the main ingredients we need, especially it introduces the 
different notions of frames for solving operator equations. By an illustrative 
example, we show that Stevenson frames seem to offer the most flexible 
concept for the discretization of operator equations. Finally, in Section
\ref{sec:revisited}, we generalize Stevenson frames to Banach 
spaces and discuss the consequences.

\section{Motivation: Solving Operator Equations} \label{sec:solvopeq0}
Let $O : \Hil \rightarrow \Hil'$ and define the bilinear form 
$a : \Hil \times \Hil \rightarrow \mathbb{R}$ by $a(u,v) = 
\left< Ou , v \right>$. Assume that $a$ satisfies the following
properties:
\begin{enumerate}
\item Let $a$ be bounded, i.e.\ there is a constant $C_S$, such that
$$ a(u,v) \le C_S \cdot \norm{\Hil}{u} \norm{\Hil}{v}.$$ 
This is equivalent to $O$ being bounded.
\item  Let $a$ be elliptic, i.e.\ there exists a constant $C_E$ 
such that 
$$ a(u,u) \ge C_E \cdot \norm{\Hil}{u}^2.$$
\end{enumerate}
Both conditions are equivalent to $O$ being bounded, 
boundedly invertible, and positive, see e.g.\ \cite{braess,brenner}.

The general goal is to find the solution $u\in\Hil$ such that
\begin{equation}
\label{sec:genweaksol1}
 a(u,v) = \ell(v)\ \text{for all}\ v \in \Hil.
\end{equation}
This is the weak formulation of the operator equation $O u = b$, 
setting $\ell(v) = \left< b , v \right>_{\Hil',\Hil}$ for $u \in \Hil$ and $b \in \Hil'$. 

In numerical approximation schemes, to get an approximate 
solution, finite dimensional subspaces $V\subset \Hil$ are 
considered and the solution $u_V\in V$ such that
\[
 a(u_V,v) = \ell(v)\ \text{for all}\ v \in V
\]
is calculated. The error between the continuous solution 
$u\in\Hil$ and the approximate solution $u_V\in V$ is 
orthogonal to the space $V$, which is known as the
Galerkin orthogonality: $a(u - u_V, v) = 0$ for all $v \in V$.
Note that, in difference to e.g.\ a Gelfand triple approach, 
the norms on $V$ and $\Hil$ are the same in the setting 
above. Instead, the Gelfand triple setting would be 
$\Hil\subset\Hil_0\subset\Hil'$ with $\norm{\Hil_0}{\cdot} 
\le c\norm{\Hil}{\cdot}$.

We shall illustrate the setting also by a practical example
from the theory of partial differential equations. To that end, 
assume that $\Omega$ is a bounded domain in $\mathbb{R}^d$
and let $\Hil_0 := L^2(\Omega)$ be the space of all square-integrable 
functions $v:\Omega\to\mathbb{R}$. As space $\Hil\subset\Hil_0$ we 
consider the Sobolov space $H_0^1(\Omega)$ which consists of all 
functions in $L^2(\Omega)$ whose first order week derivatives are 
also square-integrable and which are zero at the boundary
$\partial\Omega$. Thus, the variational formulation of 
the Poisson equation 
\[
  -\Delta u = f\ \text{in}\ \Omega, \quad
  u = 0\ \text{on}\ \partial \Omega
\]
reads
\begin{equation}\label{eq:var}
  \text{seek}\ u\in H_0^1(\Omega)\ \text{such that}\ 
  a(u,v) = \ell(v)\ \text{for all}\ v\in H_0^1(\Omega),
\end{equation}
compare \cite{braess} for example. The bilinear form 
\[
  a:H_0^1(\Omega)\times H_0^1(\Omega)\to\mathbb{R},
  \quad a(u,v) = \int_\Omega \nabla u \nabla v \d{\bf x}
\]
is continuous and elliptic due to Friedrichs' inequality,
cf.~\cite{braess}, and the linear form
\[
  \ell:H_0^1(\Omega)\to \mathbb{R},
  \quad \ell(v) = \int_\Omega f v \d{\bf x}
\]
is continuous provided that $f\in H^{-1}(\Omega)
= \big(H_0^1(\Omega)\big)'$. Hereby, the inner product 
in the pivot space $L^2(\Omega)$ is continuously extended 
onto the duality pairing $H^{-1}(\Omega) \times H_0^1(\Omega)$.
Hence, the underlying Gelfant triple is $H_0^1(\Omega)\subset 
L^2(\Omega)\subset H^{-1}(\Omega)$.

\section{Main Definitions and Notations} \label{sec:intronot0}
\subsection{Dual Pairs}
Let $X,Y$ be vector spaces and $a(x,y)$ a bilinear functional 
on $X \times Y$. Then $(X,Y)$ is called a dual pair \cite{wern1}, if 
\begin{enumerate}
\item $\forall x \in X \backslash \{ 0 \}\ \exists y \in Y\ \text{s.t.}\ a(x,y) \not=0$,
\item $\forall y \in Y \backslash \{ 0 \}\ \exists x \in X\ \text{s.t.}\ a(x,y) \not=0$.
\end{enumerate}
In short, the notation $a(x,y) = \left< x ,y \right>_a = \left< x ,y \right>$ 
is used. A classical example is a Banach space $X$ and its dual space 
$X'$. But looking at other dual pairs allows to have an explicit form 
for the dual elements \cite{heufa}.

Note that often an isomorphism is considered as an identity. For 
example, by using the Riesz mapping $\Hil \cong \Hil'$, the dual
space $\Hil'$ is often identified with $\Hil$. If two or more isomorphisms 
are involved, this identification, of course, can only be considered 
for one of those isomorphisms. For example, if we consider two 
Hilbert spaces $\Hil_1\subset\Hil_2$, the Riesz isomorphism 
can be considered only for one of them to be an identification,
see also Section \ref{sec:gelffram0}.

\subsection{Gelfand Triples}
Let $X$ be a Banach space and $\Hil$ a  Hilbert space. 
Then, the triple $(X, \Hil, X')$ is called a Banach Gelfand 
triple \cite{cofelu08}, if $X \subset\Hil\subset X'$, where 
$X$ is dense in $\Hil$,  and $\Hil$ is $w^\star$-dense in 
$X'$. The prototype of such a triple is $(\ell^1, \ell^2, 
\ell^\infty)$ in case of sequence spaces.

Note that, even if we consider the spaces all being 
Hilbert spaces --such a sequence is also called rigged 
Hilbert spaces \cite{Antoine1998}-- the Riesz isomorphism, 
in general, is not just the composition of the inclusion with 
its adjoint. This depends on the chosen concrete dual pairing. 

As another example, consider the triple $H_0^1(\Omega)\subset 
L^2(\Omega)\subset H^{-1}(\Omega)$, which has been presented 
in the practical example for the Poisson equation in Section 
\ref{sec:solvopeq0}.

\subsection{Frames} 
A sequence $\Psi = \left( \psi_k\right)_{k\in K}$ in a separable 
Hilbert space $\mathcal{H}$ is a \emph{frame} for $\mathcal{H}$, 
if there exist positive constants $A_\Psi$ and $B_\Psi$ (called lower and 
upper frame bound, respectively) that satisfy
\begin{equation} \label{sec:frambasdef1}
A_\Psi \|f\|^2\leq\sum_{k\in K}|\langle f,\psi_k\rangle|^2 \leq
B_\Psi \|f\|^2\ \text{for all}\ f \in \mathcal{H} .
\end{equation}
An upper (resp.\ lower) semi-frame is a complete system 
that only satisfies the upper (resp.\ lower) frame inequality, see 
\cite{jpaxxl09,antbal12}. A frame where the two bounds can be 
chosen to be equal, i.e.\ $A_\Psi = B_\Psi$, is called {\em tight}. We 
will denote the corresponding sequences in $\Hil$ by $\Psi =(\psi_k)_{k\in K}$ 
and $\Phi = (\phi_k)_{k\in K}$ in the following, where we consider general 
discrete index sets $K \subset\RRd$. A sequence that is 
a frame for its closed linear span is called a frame sequence.

By $C_\Psi : \Hil \rightarrow \ell^2$ 
we denote the {\em analysis operator} defined by 
$\left(C_\Psi f\right)_k = \left< f , \psi_k\right>$. The
adjoint of $C_\Psi$ is the {\em synthesis operator} 
$D_\Psi (c_k) = \sum_k c_k \psi_k$. The {\em frame 
operator} $S_\Psi = D_\Psi C_\Psi$ can be written as 
$S_\Psi f = \sum_k \left< f , \psi_k\right> \psi_k$. It is 
positive and invertible. Note that those 'frame-related' 
operators can be defined as possibly unbounded 
operators for any sequence in the Hilbert space
\cite{xxlstoeant11}.

By using the {\em canonical dual frame} $(\tilde \psi_k)$, 
i.e., $\tilde \psi_k = S_\Psi^{-1} \psi_k$ for all $k$, we 
get a reconstruction formula: 
$$f = \sum \limits_k \left< f ,
\psi_k\right> \tilde \psi_k = \sum \limits_k \left< f , \tilde 
\psi_k\right> \psi_k\ \text{for all}\ f \in \Hil.$$

The {\em Gramian matrix} $G_{\Psi}$ is defined by 
$\left(G_{\Psi}\right)_{k,l} = \left< \psi_l, \psi_k \right>$, 
also called the mass matrix. This matrix defines an operator 
on $\ell^2$ by matrix multiplication, corresponding to $G_{\Psi} 
= C_\Psi D_\Psi$. Similarily, we can define the {\em cross-Gramian 
matrix} ${\left(G_{\Psi, \Phi}\right)}_{k,l} = \left< \phi_l ,\psi_k \right>$ 
between two different frames $\Phi$ and $\Psi$. Clearly, 
$$
G_{\Psi, \Phi} c 
= \sum \limits_l \left(G_{\Psi,\Phi}\right)_{k,l} c_l 
= \left< \sum \limits_l c_l \phi_l , \psi_k \right> = C_\Psi D_\Phi c.
$$

If, for the sequence $\Psi$,  there exist constants $A_\Psi,B_\Psi >0$ 
such that the inequalities
$$ 
A_\Psi \norm{2}{c}^2 \le \norm{\Hil}{\sum \limits_{k \in K} c_k \psi_k}^2 
\le B_\Psi \norm{2}{c}^2 
$$
are fulfilled, $\Psi$ is called a {\em Riesz sequence}. If $\Psi$ 
is complete, it is called a {\em Riesz basis}.

\subsubsection{Banach Frames} \label{sec:banchfram0}
The concept of frames can be extended to Banach spaces 
\cite{Casazza2005710,chst03,gr91}: 

Let $X$ be a Banach space and $X_d$ be a Banach 
space of scalar sequences. A sequence $(\psi_k)$ in the 
dual $X^\star$ is called an $X_d$-frame for the Banach space 
$X$ $(1<p<\infty)$, if there exist constants $A_\Psi, B_\Psi>0$
such that
$$ 
A_\Psi\|f\|_{X}\leq \norm{X_d}{\psi_{k}(f)}\leq
B_\Psi\|f\|_{X}\ \textrm{for all}\
f \in X.
$$

An $X_d$-frame is called a Banach frame with respect to 
a sequence space $X_d$, if there exists a bounded reconstruction 
operator $R : X_d \rightarrow X$, such that $R \left( \psi_k(f) \right) = f$ 
for all $f \in X$. In our setting, we use $p$-frames, that is $X_d = \ell^p$
for $1\le p \le\infty$, especially, we use $X_d = \ell^2$.

A family $(g_{k})_{k\in K}\subset X$ is called a 
{\em $q$-Riesz sequence} $(1\le q \le\infty)$
for $X$, if there exist constants $A_\Psi, B_\Psi>0$ such that 
\begin{equation}\label{qRiesz} A_\Psi
\Bigg(\sum_{k\in K}|d_{k}|^{q}\Bigg)^{\frac{1}{q}}
\leq
\norm{X}{\sum_{k\in K} d_{k}g_{k}}
\leq B_\Psi\Bigg(
\sum_{k\in
K}|d_{k}|^{q}\Bigg)^{\frac{1}{q}}
\end{equation}
for all finite scalar sequence $(d_{k})$. The family 
is called a $q$-Riesz basis if it fulfills (\ref{qRiesz}) 
and $\overline{\operatorname{span}}\{g_{k}:k\in K\}=X$.

Any $q$-Riesz basis for $X^\star$ is a $p$-frame for $X$, 
where $\frac{1}{p} + \frac{1}{q} = 1$, compare \cite{chst03}.

\subsubsection{Gelfand Frames} \label{sec:gelffram0}
A frame for $\Hil$ is called a {\em Gelfand frame} \cite{dafora05} 
for the Gelfand triple $(X, \Hil, X')$ if there exists a Gelfand triple 
of sequence spaces $(X_d,\ell^2, X_d')$, such that the synthesis 
operator $D_{\Psi} : X_d\to X$ and the analysis operator 
$C_{\widetilde \Psi} : X\to X_d$ are bounded. As a result, 
see \cite{dafora05,Kapp12}, this means that $\Psi$ is a 
Banach frame for $X_d$ and $\widetilde \Psi$ a Banach 
frame for $X_d'$. 

In many approaches, see e.g.\ \cite{dafora05}, it is 
assumed for the implementation that there exists an 
isomorphism $D_B:X_d\to\ell^2$. Should $X_d$ be 
non-reflexive, then it is also assumed that $D_B^\star$ 
is an isomorphism. 
If $D_B$ is a diagonal operator, i.e.\ $D_B = \diag{w_k}$
and $D_B^{-1} = \diag{\frac{1}{w_k}}$, then 
$\Psi = \big(\frac{1}{w_k} \psi_k\big)$ is a Hilbert 
frame for $X$ and $\big(w_k\tilde\psi_k \big)$ is a 
Hilbert frame for $X'$. This is 
shown for real weights in \cite{Werner}. It is easy to 
see also for complex weights when using a weighted 
frame viewpoint \cite{xxljpa1,stoevxxl09}. These cases 
cover the weighted spaces $\ell_w^2$.

The above setting can be generalized as follows: We 
define, similar to \cite{stoev09}, the sesqui-linear form 
$\left< f , g \right>_{X}^{o} := \left< D_B C_{\widetilde \Psi} f,  
D_B C_{\widetilde \Psi} g \right>_{\ell^2}$.  It is obviously
bounded and coercive, and, in particular, $\norm{X^o}{f} 
:= \sqrt{\left< f , f \right>_{X^o}}$ is equivalent to $\norm{X}{f}$. 
Therefore, $(X,\norm{X^o}{f})$ is a Hilbert space which is isomorphic 
to $(X,\norm{X}{f})$. Now let $\xi_l := D_\Psi D_B^{-1} \delta_l$, 
where $\delta_l$ is the standard basis in $\ell^2$. This is a 
Hilbert space frame for $X$. 
Similarly, $\eta_l := D_{\widetilde \Psi} \left( D_B^\star \right)^{-1} 
\delta_l$ is a Hilbert space frame for $X'$. As a consequence 
$X$ and $X'$ are Hilbert spaces, but $X \not= X'$ and 
the inner products and the corresponding norms are 
changed, albeit equivalent to the original ones.

\subsection{Stevenson Frames}
We consider the duality $(\Hil, \Hil' )$ without using the Riesz 
isomorphism. In particular, we use the duality with respect to a second 
Hilbert space $\Hil_0$.
\begin{definition}[\cite{Stevenson03}] \label{sec:Stevfram1}
A sequence $\Psi = \left( \psi_k \right)_{k \in K} \subset \Hil$ is 
called a (Stevenson) frame for $\mathcal{H}$ 
if there exists constants 
$0 < A_\Psi \le B_\Psi < \infty$ such that  
\begin{equation}  \label{sec:eqstevfram1}
A_\Psi \cdot \norm{\Hil'}{f}^2 
\le \norm{\ell^2}{\left< f , \psi_k \right>_{\Hil',\Hil}} \le B_\Psi \cdot 
\norm{\Hil'}{f}^2 \ \text{for all}\ f \in \Hil'.
\end{equation}
\end{definition} 
Different to the Gelfand frames setting, we do not assume density. 

Typically, we consider Sobolev spaces and the $L^2$-inner 
product, which we can consider as co-orbit spaces with the 
sequence spaces  $\ell_w^2$ varying $w$. Here, invertible 
operators between different spaces exist, see Section 
\ref{sec:gelffram0}, and density is also given. In this article, 
we treat the most general setting.

In \cite{Stevenson03}, the author states ``We adapted the 
definition of a frame given in \cite[Section 3]{daubech1} by 
identifying $\Hil$ with its dual $\Hil'$ via the Riesz mapping''.  
Then, the following results are stated, also in \cite{harbr08}, 
without proofs:
The analysis operator $C_\Psi : \Hil' \rightarrow \ell^2$, 
$C_\Psi (f) = \left( \left< f, \psi_k \right> \right)_{k \in K}$ is 
bounded by \eqref{sec:eqstevfram1}, as is its adjoint 
$C_\Psi^\star : \ell^2 \rightarrow \Hil$. It can be easily shown 
that $C_\Psi^\star = D_\Psi$ is the synthesis operator with 
$D_\Psi c = \sum_{k \in K} c_k \psi_k$. Especially,
one has
$$ 
\ell^2 = \range{C_\Psi} \oplus \kernel{D_\Psi}.
$$

Define the frame operator $S_\Psi = D_\Psi C_\Psi$. It
is a mapping $S_\Psi : \Hil' \rightarrow \Hil$, which
is boundedly invertible. We can show that the sequence 
$\widetilde \Psi = \left( S_\Psi^{-1} \psi_k\right)_{k \in K}$
is a (Stevenson) $\Hil$-frame with bounds $\frac{1}{B_\Psi}$ 
and $\frac{1}{A_\Psi}$. Here, $C_{\widetilde \Psi} = C_\Psi 
S_\Psi^{-1}$ and $D_{\widetilde \Psi} = S_\Psi^{-1} D_\Psi$. 
Furthermore, it holds $S_{\widetilde \Psi} = S_\Psi^{-1}$ and, 
therefore, $S_{\widetilde \Psi} : \Hil \rightarrow \Hil'$.  

We have the reconstructions 
\begin{equation} \label{sec:reconPDE1}  
f = D_{\Psi} C_{\widetilde \Psi} h = \sum \limits_{k \in K} \left< f , \tilde \psi_k\right>_{\Hil,\Hil'} \psi_k ,
\end{equation}
and
\begin{equation} \label{sec:reconIO1}  
h = D_{\widetilde \Psi} C_\Psi h = \sum \limits_{k \in K} \left< h , \psi_k\right>_{\Hil',\Hil} \tilde \psi_k , 
\end{equation}
for all $f \in \Hil$ and $h \in \Hil'$.

The cross-Gramian matrix $G_{\Psi,\widetilde \Psi} = D_\Psi C_{\widetilde \Psi}$
is the orthogonal projection on $\range{C_\Psi}$ and coincides with 
$G_{\widetilde \Psi, \Psi}$. Therefore, $\range{C_\Psi} = \range{C_{\widetilde \Psi}}$. 

In this article, we are revisiting those statements, make them slightly 
more general, in order to make sure that not using the Riesz isomorphism is possible.

\subsection{An Illustrative Example}
Let $\Omega\subset\mathbb{R}^n$ be a sufficiently smooth,
bounded domain. We consider a multiscale analysis, i.e., a 
dense, nested sequence of finite dimensional subspaces 
\[
  V_0 \subset V_1 \subset \dots \subset V_j
  	\subset\dots\subset L^2(\Omega),
\]
consisting of piecewise polynomial ansatz functions 
$V_j = \operatorname{span}\{\varphi_{j,k}:k\in\Delta_j\}$,
such that $\dim V_j\sim 2^{jn}$ and
\[
L^2(\Omega) = \overline{\bigcup_{j\in\mathbb{N}_0}V_j},
  	\qquad V_0 = \bigcap_{j\in\mathbb{N}_0}V_j.
\]
One might think here of a multigrid decomposition of
standard Lagrangian finite element spaces or of a sequence 
of spline spaces originating from dyadic subdivision.

Trial spaces $V_j$ which are used for the Galerkin method 
satisfy typically a \emph{direct\/} or \emph{Jackson\/} estimate. 
This means that 
\begin{equation}	\label{eq:approx}
  \|v-P_j v\|_{L^2(\Omega)} \le C_J 2^{-jq}\|v\|_{H^q(\Omega)},
		\quad v\in H^q(\Omega),
\end{equation}
holds for all $0\le q\le d$ uniformly in $j$. Here, $P_j:L^2(\Omega)\to 
V_j$ is the $L^2(\Omega)$-orthogonal projection onto the trial space 
$V_j$ and $H^q(\Omega)\subset L^2(\Omega)$, $q\ge 0$, denotes 
the Sobolev space of order $q$. The upper bound $d>0$ refers in general 
to the maximum order of the polynomials which can be represented in 
$V_j$, while the factor $2^{-j}$ refers to the mesh size of $V_j$, i.e., 
the diameter of the finite elements, compare \cite{braess} for example.

Besides the Jackson type estimate \eqref{eq:approx}, there
also holds the \emph{inverse\/} or \emph{Bernstein\/} estimate
\begin{equation}	\label{eq:inverse}
  \big\|P_j v\big\|_{H^q(\Omega)}\le C_B
  	2^{jq} \big\|P_j v\big\|_{L^2(\Omega)},\quad v\in H^q(\Omega),
\end{equation} 
for all $0\le q <\gamma$, where the upper bound
\[
  \gamma:=\sup\big\{t\in\mathbb{R}: V_j\subset H^t(\Omega)\big\} > 0
\]
refers to the regularity of the functions in the trial
spaces $V_j$.
There holds $\gamma = d-1/2$ for trial functions 
based on cardinal B-splines, since they are globally 
$C^{d-1}$-smooth, and $\gamma = 3/2$ for standard 
Lagrangian finite element shape functions, since they
are only globally continuous.

A crucial requirement is the uniform frame stability of the 
systems under consideration, i.e., the existence of constants 
$A_\Phi, B_\Phi >0$ such that
\begin{equation}\label{eq:Riesz}
  A_\Phi\big\|P_j f\big\|_{L^2(\Omega)}^2
  	\le\sum_{k\in\Delta_j}|\langle f,\varphi_{j,k}\rangle|^2
	\le B_\Phi\big\|P_jf \big\|_{L^2(\Omega)}^2\ \text{for all}\ f\in L^2(\Omega)
\end{equation}
holds uniformly for all $j$. This stability is satisfied for
example by Lagrangian finite element basis functions 
defined on a multigrid hierarchy resulting from uniform 
refinement of a given coarse grid, see \cite{braess} for 
example. It is also satisfied by B-splines defined on 
a dyadic subdivision of the domain under consideration.

Having a multiscale analysis at hand, it can be used
for telescoping a given function to account for the fact 
that Sobolev norms act different on different length scales. 
Namely, the interplay of \eqref{eq:approx} and \eqref{eq:inverse} 
gives rise to the norm equivalence
\begin{equation}\label{eq:equivalence}
  \|f\|_{\widetilde{H}^{-q}(\Omega)}^2\sim\sum_{j\in\mathbb{N}_0}
  	2^{-2jq}\big\|(P_j-P_{j-1})f\big\|_{L^2(\Omega)}^2
\end{equation}
for all $0\le q<\gamma$, where $P_{-1} := 0$ and 
$\widetilde{H}^{-q}(\Omega) := \big(H^q(\Omega)\big)'$ 
denotes the dual to $H^q(\Omega)$, see \cite{DA} for 
a proof. 

In accordance with \cite{harbr08}, using \eqref{eq:Riesz}, 
we can estimate 
\begin{align*}
  \sum_{j\in\mathbb{N}_0}\sum_{k\in\Delta_j} 2^{-2jq} |\langle f,\varphi_{j,k}\rangle|^2
   &\approx \sum_{j\in\mathbb{N}_0} 2^{-2jq}\big\|P_j f\big\|_{L^2(\Omega)}^2 \\
   &= \sum_{j\in\mathbb{N}_0}
  	2^{-2jq}\sum_{\ell=0}^j\big\|(P_\ell-P_{\ell-1})f\big\|_{L^2(\Omega)}^2 \\
   &= \sum_{\ell\in\mathbb{N}_0}\big\|(P_\ell-P_{\ell-1})f\big\|_{L^2(\Omega)}^2
  	\sum_{j=\ell}^\infty 2^{-2jq}.
\end{align*}
The latter sum converges provided that $q>0$ and we 
arrive at	
\[
  \sum_{j\in\mathbb{N}_0}\sum_{k\in\Delta_j} 2^{-2jq} |\langle f,\varphi_{j,k}\rangle|^2
  	\approx \sum_{\ell\in\mathbb{N}_0} 2^{-2\ell q}\big\|(P_\ell-P_{\ell-1})f\big\|_{L^2(\Omega)}^2.
\]
In view of the norm equivalence \eqref{eq:equivalence}, 
we have thus proven that there exist constants $A_\Phi,
B_\Phi > 0$ such that
\begin{equation}	\label{eq:frame}
  A_\Phi\|f\|_{\widetilde{H}^{-q}(\Omega)}^2
  	\le\sum_{j\in\mathbb{N}_0}\sum_{k\in\Delta_j} 2^{-2jq} |\langle f,\varphi_{j,k}\rangle|^2
  	\le B_\Phi\|f\|_{\widetilde{H}^{-q}(\Omega)}^2
\end{equation}
for all $0<q<\gamma$. Therefore, in accordance with 
Definition \ref{sec:Stevfram1}, the collection
\begin{equation}\label{eq:BPX}
\Phi=\{2^{-jq}\varphi_{j,k}:k\in\Delta_j, \,j\in\mathbb{N}_0\}
\end{equation}
defines a Stevenson frame for $\mathcal{H}=H^q(\Omega)$,
where $\mathcal{H}' = \widetilde{H}^{-q}(\Omega)$ with duality 
related to $\mathcal{H}_0=L^2(\Omega)$.
Notice that this frame underlies the construction of the so-called 
BPX preconditioner, see e.g.~\cite{BPX,DA,OS2}. Especially, 
by removing all basis functions which are associated with 
boundary nodes, one gets a Stevenson frame for $\mathcal{H} 
= H_0^1(\Omega)$, as required for the Galerkin discretization 
of elliptic partial differential equations, compare 
Section~\ref{sec:solvopeq0}.

We like to emphazise that the collection \eqref{eq:BPX} 
does not define a Gelfand frame, since \eqref{eq:frame}
does not hold in $\mathcal{H}_0 = L^2(\Omega)$, i.e., 
for $q = 0$. Hence, the concept of Stevenson frames seems 
to be more flexible than the concept of Gelfand frames.

\subsection{Operator Representation in Frame Coordinates} 
\label{sec:matrrep0}
For orthonormal sequences, it is well known that operators can 
be uniquely described by a matrix representation \cite{gohberg1}. 
The same can be constructed with frames and their duals, see 
\cite{xxlframoper1,xxlgro14}. 

Let $\Psi = (\psi_k)$ be a frame in $\Hil_1$ with bounds 
$A_\Psi,B_\Psi>0$, and let $\Phi = (\phi_k)$ be a frame in 
$\Hil_2$ with $A_\Phi,B_\Phi>0$.
\begin{enumerate} 
\item Let $O : \Hil_1 \rightarrow \Hil_2$ be a 
bounded, linear operator. Thus, the infinite matrix 
$$ 
{\left( {\Mat}^{(\Phi , \Psi)} \left( O \right) \right)}_{m,n} = 
\left<O \psi_n, \phi_m \right>$$%
defines a bounded operator from $\ell^2$ to $\ell^2$ with 
$\norm{\ell^2 \rightarrow \ell^2}{\mathcal M} \le \sqrt{B_\Phi \cdot B_\Psi} 
\cdot \norm{\Hil_1 \rightarrow \Hil_2}{O}$. As an operator 
$\ell^2 \rightarrow \ell^2$, we have
$$ 
{\mathcal M}^{(\Phi , \Psi)} \left( O \right) = C_{\Phi} \circ O \circ D_{\Psi}.
$$
\item On the other hand, let $M$ be an infinite matrix defining a 
bounded operator from $\ell^2$ to $\ell^2$, $\left(M c\right)_i = 
\sum_k M_{i,k} c_k$. Then, the operator $\mathcal{O}^{(\Phi , \Psi)}$ 
defined by 
$$ 
\left( \mathcal{O}^{(\Phi , \Psi)} \left( M \right)\right) h 
= \sum \limits_k  \left( \sum \limits_j M_{k,j} \left<h, \psi_j\right> \right) \phi_k\ 
\mbox{for all}\ h \in \Hil_1
$$  
is a bounded operator from $\Hil_1$ to $\Hil_2$ with 
$$
\norm{\Hil_1 \rightarrow \Hil_2}{\mathcal{O}^{(\Phi , \Psi)} \left( M \right)} 
\le \sqrt{B_\Phi \cdot B_\Psi} \norm{\ell^2 \rightarrow \ell^2}{M}
$$
and
$$ 
\mathcal{O}^{(\Phi , \Psi)} (M) = D_{\Phi} \circ M \circ C_{\Psi} 
= \sum \limits_k  \sum \limits_j M_{k,j} \cdot \phi_k \otimes_i\psi_j.
$$
\end{enumerate} 

Please note that there is a classification of matrices that are 
bounded operators from $\ell^2$ to $\ell^2$ \cite{cron71}. 

If we start out with frames, more properties can be proved \cite{xxlframoper1}: 
Let $\Psi = (\psi_k)$ be a frame in $\Hil_1$ with bounds 
$A_\Psi,B_\Psi>0$, $\Phi = (\phi_k)$ in $\Hil_2$ with $A_\Phi,B_\Phi>0$.
\begin{enumerate}
\item It holds
$$\left( {\mathcal O^{(\Phi , \Psi)} \circ M^{(\tilde{\Phi}, \tilde{\Psi})}}\right)  
= \identity{\BL(\Hil_1,\Hil_2)} = \left( {\mathcal O^{(\tilde{\Phi}, \tilde{\Psi})} \circ M^{(\Phi , \Psi)}}\right).$$
Therefore, for all $O \in \BL(\Hil_1,\Hil_2)$:
$$ O = \sum \limits_{k,j} \left<O \tilde{\psi}_j, \tilde{\phi}_k \right>  \phi_k \otimes_i\psi_j.$$
\item  $\mathcal{M}^{(\Phi , \Psi)}$ is injective and $\mathcal{O}^{(\Phi , \Psi)}$ is surjective.
\item If $\Hil_1 = \Hil_2$, then $\mathcal{O}^{(\Psi, \tilde{\Psi})} (\identity{\ell^2}) 
= \identity{\Hil_1}$. 
\item Let $\Xi = (\xi_k)$ be any frame in $\Hil_3$, and 
$O: \Hil_3 \rightarrow \Hil_2$ and $P: \Hil_1 \rightarrow \Hil_3$. Then, it holds
$$ \mathcal{M}^{(\Phi, \Psi)}\left( O \circ P \right) = \left( \mathcal{M}^{(\Phi, \Xi)}\left( O \right) \cdot \mathcal{M}^{(\tilde{\Xi}, \Psi)} \left( P \right) \right).$$
\end{enumerate}

Note that, in the Hilbert space of Hilbert-Schmidt operators, the
tensor product $\Psi \otimes \Phi := \left\{ \psi_k \otimes \psi_l \right\}_{(k,l) \in K \times K}$ is a Bessel sequence / frame 
sequence / Riesz sequence, if the starting sequences $\Psi$ and $\Phi$ are 
\cite{xxlframehs07}, with ${\mathcal M}^{(\Phi , \Psi)}$ being 
the analysis and $\mathcal O^{(\Phi , \Psi)} $ being the 
synthesis operator. This relation is even an equivalence 
\cite{bo08-1}. 

For the invertibility, it can be shown \cite{xxlrieck11,xxlriek11}: 
If and only if $O$ is bijective, then $M=\Mat^{(\Phi,\Psi)}(O)$ is 
bijective as operator from $\range{C_{\Psi}}$ onto $\range{C_{\Phi}}$. 
In this case, one has
$$
M^{\dagger}= \Mat^{(\tilde \Psi, \tilde \Phi)} 
(O^{-1})= G_{\tilde \Psi, \tilde \Phi}\circ \Mat^{(\Phi,\Psi)}\left(O^{-1}\right)
G_{\tilde \Psi, \tilde \Phi} = \Mat^{(\Psi,\Phi)}\left(S_\Psi^{-1} O^{-1} S_\Phi^{-1}\right). 
$$

If we have an operator equation $O u = b$, we use
$$ 
O u = b \Longleftrightarrow \sum \limits_k \left< u, \tilde \psi_k\right> O \psi_k = b,
$$
which implies
$$
 \sum \limits_k \left< u, \tilde \psi_k\right> \left
	< O \psi_k , \psi_l \right> = \left< b , \psi_l \right>
$$
for all $l\in K$. Setting ${\bf M} = \Mat^{(\Psi,\Psi)}(O) $, 
$\vec{u} = C_{\widetilde \Psi} u$ and $\vec{b} = C_{\Psi} b$, 
we thus have 
$$ 
O u = b \Longleftrightarrow {\bf M} \vec{u} = \vec{b}.
$$

Note that, for numerical computations, see e.g.\ \cite{dafora05,Stevenson03}, 
the system of linear equations ${\bf M} \vec{u} = \vec{b}$ is solved. Then, 
$u = D_\Psi \vec{u}$ is the solution to $O u = b$, avoiding the numerically expensive calculation of a dual frame \cite{xxlfei1,janssoend07,perxxl17}. 
If the frame is redundant, then $u_k$ can be different to $\big\langle u, 
\tilde \psi_k\big\rangle$. If a Tychonov regularization is used, we obtain
$u_k =\big\langle u, \tilde \psi_k\big\rangle$ by \cite[Prop.\ 5.1.4]{gr01}.

\section{Stevenson Frames Revisited}\label{sec:revisited}
As some of the references dealing with Stevenson frames 
used an unlucky formulation, when stating if or if not the Riesz 
isomorphism is used, see e.g.~\cite{dafora05,Stevenson03},
the authors decided to check everything again, 
and pay particular attention to the avoidance of the Riesz 
isomorphism, i.e.\ to not use $\Hil \cong \Hil'$.  

To {\em not} use the Riesz isomorphism in a treatment of 
Hilbert spaces is mind-boggling, so we decided to use Banach 
spaces, to be sure to avoid all pitfalls. (Note, however, that the 
Riesz isomorphism will be used on the sequence space $\ell^2$.) 
In particular, this is a generalization of the original definition. 
The used spaces are necessarily isomorphic to Hilbert 
spaces, but not Hilbert spaces per se.

\subsection{Stevenson Banach Frames} \label{sec:StevBanach0}
We start out with a generalized definition. (We will show that 
this is isomorphic, but {\em not\/} identical to the original definition.)
\begin{definition} \label{sec:StevBanachfram1} Let $(X,X')$ be a 
dual pair of reflexive Banach spaces. Let $\Psi = (\psi_k)_{k \in K} 
\subset X$. It is called a Stevenson Banach frame for $X$,
if there exist bounds $0 < A_\Psi \le B_\Psi < \infty$ such that
$$
A_\Psi \norm{X'}{f}^2 \le \norm{\ell^2}{\left< f , \psi_k \right>_{X',X}} 
\le B_\Psi \norm{X'}{f}^2 \text{ for all } f \in X'.
$$ 
\end{definition}

The analysis operator 
$$C_\Psi : X' \rightarrow \ell^2, \quad 
C_\Psi (f) = \left( \left< f, \psi_k \right>_{X',X} \right)_{k \in K}$$ 
is bounded by $\sqrt{B_\Psi}$ by definition. (Note that we use here the notation which is
more common for Banach spaces \cite{gr91}.) As a consequence
of the open mapping theorem, $C_\Psi$ is one-to-one and has 
closed range. 

For $d = (d_k) \in \ell^2(K)$ with finitely many non-zero 
entries, i.e.\ $d \in c^{00}$, consider
$$ 
\left< C_\Psi f , d \right>_{\ell^2} = \sum \limits_{k \in K}  
\left< f, \psi_k \right>_{X',X} d_k =
 \left< f, \sum \limits_{k \in K}  d_k \psi_k \right>_{X',X}. 
$$
By using a standard density argument and the reflexivity,
it can easily be shown that $C_\Psi^\star = D_\Psi$, where $D_\Psi: 
\ell^2 \rightarrow X$ is the synthesis operator with $D_\Psi c = 
\sum_{k \in K} c_k \psi_k$. The bound of $D_\Psi$ is also $\sqrt{B_\Psi}$.
The sum converges unconditionally. Indeed, consider $c \in \ell^2$. Then,
let $K_0 \subset K$ be a finite set, such that 
$$
\sum \limits_{k \not \in K_0} \left| c_k \right|^2 < \epsilon' := \frac{\epsilon}{\sqrt{B_\Psi}}.
$$
For another finite index set $K_1\supset K_0$, we thus find 
$$ \norm{\Hil}{\sum \limits_{k \not \in K} c_k \psi_k - \sum \limits_{k \not \in K_1} c_k \psi_k} 
= \norm{\Hil}{D_\Psi \left( c - c \cdot \chi_{K_1} \right)} < \sqrt{B_\Psi} \epsilon' = \epsilon.
$$ 
Hence, by e.g.\ \cite[IV.5.1]{wern1} and the fact that $\ell^2$ 
is a Hilbert space, we deduce
\begin{equation} \label{sec:orthcompl1} 
\ell^2 = \range{C_\Psi} \oplus \kernel{D_\Psi}.
\end{equation}

We define the frame operator $S_\Psi = D_\Psi C_\Psi$, 
which is a mapping $S_\Psi : X' \rightarrow X$. In particular,
the operator $S_\Psi$ is self-adjoint. By definition of $S_\Psi$,
it follows that
\begin{equation} \label{sec:framsesquiA} 
\left< S_\Psi f, g\right>_{X,X'} \le \norm{X}{S_\Psi f} \norm{X'}{g} 
\le B_\Psi \cdot \norm{X'}{f} \cdot \norm{X'}{g}. 
\end{equation}
Hence, $S_\Psi$ is bounded with bound $B_\Psi$. Furthermore,
we have
\begin{equation} \label{sec:framsesquiB}
\left< S_\Psi f, f \right>_{X,X'} = \left< C_\Psi^\star C_\Psi f, f \right>_{X,X'} =
\left< C_\Psi f, C_\Psi f \right>_{\ell^2} = \norm{\ell^2}{C_\Psi f}^2 \ge A_\Psi \cdot \norm{X'}{f}^2,
\end{equation}
which implies that $S_\Psi$ is one-to-one and positive. By \cite[IV.5.1]{wern1},
this also means that $S_\Psi^\star = S_\Psi$ has dense range.
$S_\Psi$ also has a bounded inverse since
$$ 
\norm{X}{S_\Psi f} = \sup \limits_{\begin{smallmatrix} \norm{X'}{g} = 1 \\ 
g \in X'\end{smallmatrix}} \left< S_\Psi f, g \right>_{X,X'} 
\ge \left< S_\Psi f, \frac{f}{\norm{X'}{f}} \right>_{X,X'} \ge A_\Psi \cdot \norm{X'}{f}.
$$
Therefore, it has closed range \cite[Theorem XI.2.1]{gohbgol1}. 
Consequently, $S_\Psi$ is onto and bijective with 
$$
A_\Psi \norm{X'}{f} \le \norm{X}{S_\Psi f} \le B_\Psi \norm{X'}{f}.
$$
Thus, $S_\Psi^{-1}$ is also self-adjoint, and 
\begin{equation} \label{sec:ellcandual1}
\frac{1}{B_\Psi} \norm{X}{g} \le \norm{X'}{S_\Psi^{-1} g} \le \frac{1}{A_\Psi} \norm{X}{g}.
\end{equation}

\begin{theorem} \label{sec:dualfram1} 
The sequence $\widetilde \Psi = \big( \tilde \psi_k\big)_{k \in K} 
:= \left( S_\Psi^{-1} \psi_k\right)_{k \in K} \subset X'$ is a Stevenson Banach 
frame for $X'$ with bounds $\frac{1}{B_\Psi}$ and $\frac{1}{A_\Psi}$. The range 
of its analysis operator coincides with the one of the primal frame, i.e.\ 
$\range{C_\Psi} = \range{C_{\widetilde \Psi}}$. The related operators are 
$C_{\widetilde \Psi} = C_\Psi S_\Psi^{-1}$, $D_{\widetilde \Psi} = S_\Psi^{-1} D_\Psi$ 
and $ S_{\widetilde \Psi} =  S_{\Psi}^{-1}$. For $f \in X$ and $g \in X'$, 
we have the reconstructions 
$$ 
f = \sum \limits_{k \in K} \left< f , \tilde \psi_k\right>_{X,X'} \psi_k 
\quad\text{and}\quad 
g =  \sum \limits_{k \in K} \left< g , \psi_k\right>_{X',X} \tilde \psi_k.
$$
\end{theorem} 

\begin{proof}
It obviously holds $S_\Psi^{-1} \psi_k \in  X'$. Moreover, we have
on the one hand
\begin{align*}\sum \limits_{k \in K} \left| \left< f , \tilde \psi_k  \right>_{X,X'} \right|^2 
&= \sum \limits_{k \in K} \left| \left< f , S_\Psi^{-1} \psi_k  \right>_{X,X'} \right|^2 =
\sum \limits_{k \in K} \left| \left< S_\Psi^{-1} f ,  \psi_k  \right>_{X',X} \right|^2\\
&\le B_\Psi \norm{X'}{S_\Psi^{-1} f}^2 \le \frac{B_\Psi}{A_\Psi^2} \norm{X}{f}^2
\end{align*}
and on the other hand
$$ 
\sum \limits_{k \in K} \left| \left< f , \tilde \psi_k  \right>_{X,X'} \right|^2 
\ge A_\Psi \norm{X'}{S_\Psi^{-1} f}^2 \ge \frac{A_\Psi}{B_\Psi^2} \norm{X'}{f}^2.
$$
Hence, $\widetilde \Psi$ is an $X'$-frame.
By employing the invertibility of $S_\Psi$ for $g = S_\Psi^{-1} f$, we get
$$ \left< f , S_\Psi^{-1} \psi_k \right>_{X,X'} 
= \left< S_\Psi g , S_\Psi^{-1} \psi_k \right>_{X,X'}
= \left< g , S_\Psi S_\Psi^{-1} \psi_k \right>_{X',X} 
= \left< g , \psi_k \right>_{X',X}.$$
This implies $\range{C_\Psi} = \range{C_{\widetilde \Psi}}$,
where $C_{\widetilde \Psi} = C_\Psi S_\Psi^{-1}$ and $D_{\widetilde \Psi} 
= S_\Psi^{-1} D_\Psi$. Furthermore, $S_{\widetilde \Psi} = S_\Psi^{-1}:X\rightarrow X'$, 
because it holds
\begin{align*}
  S_{\widetilde \Psi} f 
  &= \sum \limits_k \left< f , S_\Psi^{-1} \psi_k \right>_{X,X'} S_\Psi^{-1} \psi_k\\
  &= S_\Psi^{-1} \sum \limits_k \left< S_\Psi^{-1} f , \psi_k \right>_{X,X'}  \psi_k 
  = S_\Psi^{-1} S_\Psi S_\Psi^{-1}f = S_\Psi^{-1}f
\end{align*}
for all $f\in X$.

Finally, we have the reconstructions 
\[
f = D_{\Psi} C_{\Psi} S_{\Psi}^{-1} f 
  = D_{\Psi} C_{\widetilde \Psi} f 
  = \sum \limits_{k \in K} \left< f , \tilde \psi_k\right>_{X,X'} \psi_k
\]
for all $f\in X$ and
\[
g = S_{\Psi}^{-1} D_{\Psi} C_{\Psi}   
   = D_{\widetilde \Psi} C_\Psi g 
   = \sum \limits_{k \in K} \left< h , \psi_k\right>_{X',X}\tilde \psi_k
\]   
for all $g \in X'$.

As 
$$
\left< S_\Psi^{-1} x,x \right>_{X',X} \leq \norm{X'}{S_\Psi^{-1} x} 
\norm{X}{x} \leq \frac{1}{A_\Psi} \norm{X}{x}^2, 
$$
we have the sharper upper bound. On the other hand, since
$\left< S_\Psi^{-1}\cdot,\cdot\right>$ defines a positive sesqui-linear 
form, the Cauchy-Schwarz inequality implies
$$
\left| \left< S_\Psi^{-1} x,y \right>_{X',X} \right|^2 
\le \left< S_\Psi^{-1} x,x \right>_{X',X}  \left< S_\Psi^{-1} y,y \right>_{X',X}.
$$
Thus, with $x = S_\Psi u$, there holds
$$
\left| \left< u,y \right>_{X',X} \right|^2 
\le \left< u, S_\Psi u \right>_{X',X}  \left< S_\Psi^{-1} y,y \right>_{X',X}
$$ 
and consequently
$$ 
\norm{X}{y}^2 = \sup \limits_{\begin{smallmatrix} \norm{X'}{u} = 1 \\ 
u\in X'\end{smallmatrix}} \left| \left< u,y \right>_{X',X} \right|^2 
\le B_\Psi  \left< S_\Psi^{-1} y,y \right>_{X',X}.
$$
So, the sharper bounds $\frac{1}{B_\Psi}$ and $\frac{1}{A_\Psi}$ follow.
\end{proof}

The fact that $\range{C_\Psi} = \range{C_{\widetilde \Psi}}$ 
is very different to the Gelfand frame setting, where the ranges 
$\range{{C_\Psi}_{|_{X'}}} \not= \range{{C_{\widetilde \Psi}}_{|_{X}}}$ 
even live in different sequence spaces.

\begin{theorem} The cross-Gramian matrix $G_{\Psi,\widetilde \Psi} 
= C_\Psi D_{\widetilde \Psi}$ is the orthogonal projection on 
$\range{C_\Psi}$ and coincides with $G_{\widetilde \Psi, \Psi}$. 
\end{theorem} 
\begin{proof} 
We have that the cross-Gramian matrix of a frame and its dual is a projection:
$$\left(G_{\Psi,\widetilde \Psi}\right)^2  
=  C_\Psi D_{\widetilde \Psi}  C_\Psi D_{\widetilde \Psi} 
=  C_\Psi D_{\widetilde \Psi}.= G_{\Psi, \widetilde \Psi}.$$
Next, it holds
$$G_{\Psi,\widetilde \Psi}^\star = \left( C_\Psi D_{\widetilde \Psi} \right)^\star 
= C_{\widetilde \Psi}  D_\Psi = G_{\widetilde \Psi,\Psi}.$$
In addition, since
$$ {\Big( G_{\Psi,\widetilde \Psi}\Big)}_{k,l} = \left< \tilde \psi_l, \psi_k\right>_{X',X} 
= \left< S_\Psi^{-1} \psi_l, \psi_k\right>_{X',X} =  \left< \psi_l, S_\Psi^{-1} \psi_k\right>_{X,X'} 
= {\Big( G_{\widetilde  \Psi,\Psi}\Big)}_{k,l}, $$
we conclude $G_{\Psi,\widetilde \Psi} = G_{\widetilde \Psi,\Psi}$.
Thus, $G_{\Psi,\widetilde \Psi}$ is self-adjoint.
\end{proof}

\begin{theorem}
The collection $\Psi$ is a Stevenson Banach frame for $X$ 
with bounds $A_\Psi$ and $B_\Psi$ if and only if
$$ 
\frac{1}{B_\Psi} \norm{X}{f} 
\le \inf \limits_{d \in \ell^2,\,D_\Psi d = f} \norm{\ell^2}{d} 
\le \frac{1}{A_\Psi} \norm{X}{f}.
$$
In particular, for any $f \in X $ with $f = \sum_{k \in K} d_k \psi_k$ 
and $d= (d_k) \in \ell^2$, we have $\norm{\ell^2}{d} \ge 
\norm{\ell^2}{C_{\widetilde \Psi} f}$.
\end{theorem}

\begin{proof} 
Given 
$$
f = \sum_{k \in K} d_k \psi_k\in X,
$$
we have the representation 
$$
f = \sum_{k \in K} \left< f, \tilde \psi_k \right>_{X,X'} \psi_k.
$$ 
Hence, 
$$
\left(d_k - \left< f, \tilde \psi_k \right>_{X,X'}\right)\in\kernel{D_\Psi}.
$$
By \eqref{sec:orthcompl1} and Theorem \ref{sec:dualfram1}, there 
follows $\norm{\ell^2}{d} \ge \norm{\ell^2}{C_{\widetilde \Psi} f}$.
\end{proof}

Consequently, a Stevenson frame is a Riesz basis for $X$ if and only 
if $D_\Psi$ is one-to-one.

\subsection{Is $X'$ a Hilbert Space?}
Set $\left< u , v \right>_{X'_{\Hil}} :=  \left< u, S_\Psi v \right>_{X', X}$. This 
is, trivially, a symmetric and positive bilinear form by above and, therefore, 
an inner product on $X'$. Hence, $X'$ is a pre-Hilbert space with this inner 
product. By \eqref{sec:framsesquiA} and \eqref{sec:framsesquiB}, the 
corresponding norm is equivalent to the original one as $\sqrt{A_\Psi} \norm{X'}{f}\le 
\norm{X'_{\Hil}}{f} \le \sqrt{B_\Psi} \norm{X'}{f}$. Thus, $(X', \left<\cdot,\cdot\right>)$
is a Hilbert space. 

Note that, in particular for numerics, it is sometimes not enough 
to consider equivalent norms. While well-posed problems stay 
well-posed for equivalent norms, this becomes important for 
concrete implementations, as things like condition numbers, 
constants in convergence rates, etc.\ are considered. 

From a frame theory perspective, switching to an equivalent 
norm can destroy or create tightness, in particular, the switch 
from one norm to the other changes the frame bound ratio 
$\frac{B_\Psi}{A_\Psi}$. We refer the reader to e.g.\ weighted and controlled frames 
\cite{xxljpa1}, which are under very mild conditions equivalent to classical 
Hilbert frames. Nonetheless, they have applications for example in the 
implementation of wavelets on the sphere \cite{bogdvan1,jacques-th}, 
and nowadays become important for the scaling of frames \cite{Casazza2017,kutoko13}. 
As a trivial example, look at $\Psi := \left\{ e_1, e_1, e_2, e_2, e_3, e_3,\dots \right\}$, 
where $E = \left\{ e_i \right\}_{i\in\NN}$ is an orthonormal basis for $\Hil$. 
Then, $\Psi$ is a tight frame with $A_\Psi = 2$. Looking at the reweighted 
version $\Phi := \left\{2 e_1, 2 e_1, e_2/2, e_2/2, 2 e_3, 2 e_3, \dots \right\}$, 
we loose tightness, since this frame has bounds $A=1$ and $B=4$. Note that 
there exists an invertible bounded operator that maps the single 
elements from $\Psi$ into $\Phi$, i.e they are equivalent sequences \cite{Casaz1}.

Also note that, if it does not make sense to assume that $X\subset X'$, then $\Psi$ 
cannot be a Hilbert space frame per se. This can only be true for $\Psi' := 
( \psi'_k) = (I \psi_k)$, where $I$ is an isomorphism from $X$ to $X'$, 
for example, choosing $I = S_\Psi^{-1}$. In this case, the frame bounds 
are preserved, but the roles of primal and dual frames interchange.

This especially means that, if the frame bound ratio is important, 
distinguishing $\ell^2$-Banach frames from Hilbert frames is necessary, 
especially if concrete examples for $X$ and $X'$ are used, 
where an identification is not possible, i.e.\ $X \not=X'$. 
As such, Definition \ref{sec:StevBanach0} is, of course, 
equivalent to the standard frame definition for Hilbert spaces, 
but {\em the frame bound ratio changes}.

We like to remark that, by using the dual frame, one can 
also conclude that $X$ itself is a Hilbert space.

\subsection{Matrix Representation}
Let us also revisit the statements about the matrix representation 
of operators \cite{harbr08,Stevenson03}. To this end, let $\Psi$ 
be Stevenson Banach frame for $X$.

Let us now consider an operator $O : X \rightarrow X'$ and 
define 
$$
{\left( {\Mat}^{(\Psi)} \left( O \right) \right)}_{m,n} = 
\left<O \psi_n, \psi_m \right>_{X',X}.
$$ 
Then, ${\Mat}^{(\Psi)} \left( O \right) =  
C_{\Psi} O  D_{\Psi}$, which implies
$$
\norm{\ell^2 \rightarrow \ell^2}{{\Mat}^{(\Psi)} \left( O \right)} 
\le B_\Psi \norm{X\rightarrow X'}{O}. 
$$ 
(As in Section \ref{sec:matrrep0}, we could consider different 
sequences, and the arguments would still work, but following 
the argument in the introduction and for easy reading we will not.)

For an invertible operator $O$, we have
$$
\Mat^{(\widetilde \Psi)} (O^{-1})  \Mat^{(\Psi)} (O) 
= C_{\widetilde \Psi}  O^{-1}  D_{\widetilde \Psi} C_{\Psi} O  D_{\Psi} 
= G_{\widetilde \Psi,\Psi}.
$$ 
(For the analogue result in the Hilbert frame case, see \cite{xxlrieck11,xxlriek11}.)
Equivalently, 
$$
\Mat^{(\Psi)} (O) \Mat^{(\widetilde \Psi)}(O^{-1}) = G_{\widetilde \Psi, \Psi}.
$$
Therefore, as $G_{\widetilde\Psi, \Psi}$ is the orthogonal projection on $\range{C_\Psi}$ the operator
$\Mat^{(\Psi)} (O)_{|_{\range{C_\Psi}}}$ is boundedly invertible, 
as 
$$
\norm{\range{C_\Psi} \rightarrow \range{C_\Psi}}{ \Mat^{(\Psi)} (O)}
\ge A_\Psi\norm{X'\rightarrow X}{O^{-1}}^{-1}.
$$
Furthermore, 
$$
\kernel{\Mat^{(\widetilde \Psi)} (O)} = \kernel{D_\Psi}.
$$ 
If $O$ is symmetric, then ${\Mat}^{(\Psi)} \left( O \right)$ is 
symmetric. If $O$ is non-negative, so is ${\Mat}^{(\Psi)} \left( O \right)$. 
In particular, we have now have settled all statements in \cite{harbr08,Stevenson03}.

\section*{Acknowledgment}
This research was supported by the START project FLAME Y551-N13 
of the Austrian Science Fund (FWF) and the DACH project BIOTOP 
I-1018-N25 of the Austrian Science Fund (FWF) and 200021E-142224
of the Swiss National Science Foundation (SNSF). 

The authors like to thank Stephan Dahlke, Wolfgang Kreuzer, 
and Diana Stoeva for fruitful discussions.

\bibliographystyle{amsplain}

\providecommand{\bysame}{\leavevmode\hbox to3em{\hrulefill}\thinspace}
\providecommand{\MR}{\relax\ifhmode\unskip\space\fi MR }
\providecommand{\MRhref}[2]{%
  \href{http://www.ams.org/mathscinet-getitem?mr=#1}{#2}
}
\providecommand{\href}[2]{#2}

\end{document}